\documentclass{elsarticle}
\usepackage[all]{xy}
\usepackage{amscd}
\newcommand{\ds}{\displaystyle}
\newcommand{\lra}{\longrightarrow}
\renewcommand{\P}{\frak{P}}

\usepackage{color}
\usepackage{amssymb}
\usepackage{amsthm}
\usepackage{amsmath}
\usepackage{latexsym}

\newtheorem{Thm}{Theorem}[section]

\newtheorem{Cor}[Thm]{Corollary}

\newtheorem{Lem}[Thm]{Lemma}

\newtheorem{Exa}[Thm]{Example}

\newcommand{\bQ}{\overline{\mathbb{Q}}}
\newcommand{\Q}{\mathbb{Q}}
\newcommand{\Z}{\mathbb{Z}}

\renewcommand{\O}{{\cal O}}

\renewcommand{\ker}[1]{{\rm ker}\hspace{0.1cm}{#1}}

\renewcommand{\P}{{\frak P}}
\newcommand{\p}{{\frak p}}

\newcommand{\mapright}[1]{%
  \smash{\mathop{%
  \hbox to 1cm{\rightarrowfill}}\limits^{#1}}}
\newcommand{\mapleft}[1]{%
  \smash{\mathop{%
  \hbox to 1cm{\leftarrowfill}}\limits^{#1}}}

\newcounter{CounterEQUlabel}
\newcommand{\EQUlabel}[1]{\label{#1}
\ifcase \theCounterEQUlabel
   \relax
  \or
   \hspace{1em}
   \mbox{{\color{blue} \tiny$\langle$\rmfamily#1$\rangle$}}\index{zzz#1@#1}
  \fi
}

\title{On the class numbers of the fields of the 
$p^n$-torsion points of certain elliptic curves over $\Q$}

\author[rvt]{Fumio SAIRAIJI
}
\ead{sairaiji@it.hirokoku-u.ac.jp}

\author[focal]{Takuya YAMAUCHI\fnref{fn2}
}
\ead{yamauchi@edu.kagoshima-u.ac.jp}

\fntext[fn2]{The second author is partially supported by JSPS Grant-in-Aid for Scientific
Research No.23740027.}
\address[rvt]{
Faculty of Engineering,\\
Hiroshima International University,\\
Hiro, Hiroshima\\
737-0112, 
JAPAN\\
e-mail:~sairaiji@it.hirokoku-u.ac.jp}
\address[focal]{
Department of mathematics, Faculty of Education\\
Kagoshima University\\
Korimoto 1-20-6 Kagoshima 890-0065, 
JAPAN \\
and 
Department of mathematics \\
 University of Toronto \\
Toronto, Ontario M5S 2E4, CANADA
\\
e-mail:~yamauchi@edu.kagoshima-u.ac.jp \\
or tyama@math.toronto.edu}

\begin{document}

\begin{abstract}
Let $E$ be an elliptic curve over $\Q$ with prime conductor $p$. 
For each non-negative integer $n$ we put $K_n:=\Q(E[p^n])$. 
The aim of this paper is to estimate the order of the $p$-Sylow group of the ideal class group of $K_n$. 
We give a lower bound in terms of the Mordell-Weil rank of $E(\Q)$. 
As an application of our result, 
we give an example such that 
$p^{2n}$ divides the class number of the field $K_n$ in the case of $p=5077$ for each positive integer $n$. 
\end{abstract}
\begin{keyword}
elliptic curve, Mordell-Weil rank, class number.
\end{keyword}

\maketitle
\section{Introduction}
Let $E$ be an elliptic curve over $\Q$ with complex multiplication by the ring of integers $\O_K$ of an imaginary quadratic field $K$. 
Let $p$ be an odd prime number such that $E$ has good ordinary reduction 
at $p$. 
Then $p$ splits completely in $K$ with $p\O_K=\pi\overline{\pi}$, where $\pi$ is a prime element of $K$. 
Let $K_1$ be the Galois extension generated by the $\pi$-torsion points on $E$ and $K_\infty$ the $\Z_p$-extension of $K_1$ generated by the $\pi$-power torsion points on $E$. 
Then Greenberg (\cite{G}, p.\ 365) shows that 
$K_\infty$ has $\lambda$-invariant greater than or equal to $r-1$, where $r$ is the $\O_K$-rank of $\Q$-rational points $E(\Q)$. 

There are at least two possible natural directions in the generalizations of the results of Greenberg. 
The first generalization is on the higher dimensional CM abelian varieties, 
which is obtained by Fukuda-Komatsu-Yamagata \cite{FKY}. 
The second generalization is on non-CM elliptic curves, 
which is the aim of this paper. 
\par
\bigskip
Let $p$ be an odd prime number. 
Let $E$ be an elliptic curve over $\Q$ with conductor $p$. 
We denote the minimal discriminant of $E$ by $\Delta$. 
This condition forces us that $p\ge 11$, $p\neq 13$ (cf.\ \cite{cre}), 
and that $\Delta$ divides $p^5$ (cf.\ \cite{MO}). 
For each non-negative integer $n$, 
we denote the $p^n$-torsion subgroup of $E$ by $E[p^n]$. 
We also denote the field generated by the points on $E[p^n]$ over $\Q$ by $K_n$. 
The Mordell-Weil Theorem asserts that $E(\Q)$ is a finitely generated 
abelian group. 
Thus there exists a free abelian subgroup $A$ of $E(\Q)$ of finite rank such that $A+E(\Q)_{{\rm tors}}=E(\Q)$. We denote the rank of $A$ by $r$. 
We put $L_n:=K_n([p^n]_E^{-1}A)$, 
where $[p^n]_E$ is the multiplication-by-$p^n$ map on $E$. 

We denote the generators of $A$ by $P_1,\ldots,P_{r}$. 
For each $j\in \{1,\ldots,r\}$ we take a point $T_j$ of $E(L_n)$ satisfying 
$$
[p^n]_E(T_j)=P_j.
$$
Then we have $L_n=K_n(T_1,\ldots,T_{r})$. 
We have an injective homomorphism 
$$
\Phi_n:~\mbox{Gal}(L_n/K_n) \rightarrow E[p^n]^{r}~:~
\sigma \mapsto ({}^\sigma T_j- T_j)_j.
$$
In particular, the extension degree $[L_n:K_n]$ is equal to a power of $p$. 

We denote the maximal unramified abelian extension of $K_n$ 
by $K_n^{{\rm ur}}$. 
We define the exponent $\kappa_n$ by 
$$
[L_n \cap K_n^{{\rm ur}}:K_n]=p^{\kappa_n}.
$$
Then we have the following theorem. 
%
%
%
%
\begin{Thm}
Let $E$ be an elliptic curve with prime conductor $p$. 
Then the inequality 
\begin{equation}
\kappa_n \geq (2r-4)n
\label{120307b}
\end{equation}
holds for each $n \geq 1$. 
\label{1002a}
\end{Thm}
\par
\bigskip
The organization of this paper is as follows. In \S 2, 
we briefly recall Bashmakov's result \cite{bas} from Lang \cite{L} and we investigate the degree $[L_n:K_n]$ of the extension $L_n$ over $K_n$. 
In \S 3, we investigate the degree of the $p$-adic completion of $L_n$ over the one of $K_n$, 
which is used for the estimate of the inertia group in \S. 4. 
We give the proof of Theorem \ref{1002a} in \S 4, 
and we give a numerical example of $\kappa_n$. 
In \S 5, 
we note the our result does not come from class number formula for the cyclotomic $\Z_p$-extension of $K_1$.
\par
\bigskip
The authors gratefully acknowledge helpful discussions with 
Professor K.\ Matsuno. 
The authors also wish to thank the referee for useful comments. 
\section{The extension $L_n$ over $K_n$}

In this section, we collect some results of Bashmakov \cite{bas} 
from Lang \cite{L}. 
We add some detail to the proof of  Lemma 1 of \cite[p.\ 117]{L}, which is one of the key lemmas.

Let the notations be the same as \S 1. 
Since $E$ is an elliptic curve with prime conductor $p$, 
$E$ is semistable and $p\ge 11$. 
Since $2$ is the minimal good prime of $E$ and $p>(\sqrt{2}+1)^2$, 
we have $\mbox{{\rm Gal}}(K_1/\Q)\simeq \mbox{{\rm GL}}_2(\Z/p\Z)$ 
by Corollaire 1 of Serre \cite{serre}. 
Since $p\ge 11$, 
${\rm Gal}(K_1/\Q) \simeq {\rm GL}_2(\Z/p\Z)$ implies ${\rm Gal}(K_n/\Q) \simeq {\rm GL}_2(\Z/p^n\Z)$ for all $n\ge 1$ (cf.\ \cite{serre}, IV, 3.4, Lemma 3).

We put $K_\infty:=\cup_{n \geq 1}K_n$ and $L_\infty:=\cup_{n \geq 1}L_n$. 
We have $\mbox{Gal}(K_\infty/\Q)\simeq\mbox{GL}_2(\Z_p)$ and $\{\Phi_n\}_{n\ge 1}$ induces an injective homomorphism 
$$
\Phi_\infty:~\mbox{Gal}(K_\infty/\Q)\rightarrow T_p(E)^r,
$$
where $T_p(E)$ is the $p$-adic Tate module. 

We put $H_n:=1+p^nM_2(\Z_p)\subset \mbox{GL}_2(\Z_p)$ for each $n\ge 1$. 
$K_n$ is the fixed subfield of the group $H_n$. 
%
%
%
%
\begin{Lem}
The equation $H_n^{p^n}=H_{2n}$ holds for each $n\ge 1$. 
\label{0812b}
\end{Lem}
\begin{proof}
For each $M$, $M^\prime$ in $M_2(\Z_p)$, 
we have 
$$
(1+p^nM)(1+p^nM^\prime)\equiv 1+p^n(M+M^\prime)\bmod{p^{2n}}.
$$
It follows from $(1+p^nM)^{p^n}\equiv 1\bmod{p^{2n}}$ that $H_n^{p^n} \subset H_{2n}$. 

For each $M$ in $M_2(\Z_p)$, we formally put 
\begin{equation}
(1+p^{2n}M)^{\frac{1}{p^n}}:=
\sum_{j\ge 0}
\begin{pmatrix}{\frac{1}{p^n}}\\j\end{pmatrix}
(p^{2n}M)^{j}
=1+p^nM+\frac{1-p^n}{2}p^{2n}M^2+\cdots
\label{0812a}
\end{equation}
where $\begin{pmatrix}{\frac{1}{p^n}}\\j\end{pmatrix}=\ds\frac{1}{j!}\prod_{i=0}^{j-1}(\frac{1}{p^n}-i)$. 
We see that 
$$
\mbox{{\rm ord}}_p(j!)=\sum_{i\ge 1}\bigg[\frac{j}{p^i}\bigg]
< \sum_{i \ge 1}\frac{j}{p^i}=\frac{j}{p-1},
$$
where $[\frac{j}{p^i}]$ is the maximal integer less than or equal to $\frac{j}{p^i}$. 
The coefficient of $M^j$ is equal to 
$$
\begin{pmatrix}{\frac{1}{p^n}}\\j\end{pmatrix}p^{2nj}
=\frac{1\cdot(1-p^n)\cdots(1-(j-1)p^n)}{j!}p^{nj}
$$
and we have
$$
\mbox{{\rm ord}}_p\bigg(\begin{pmatrix}{\frac{1}{p^n}}\\j\end{pmatrix}p^{2nj}\bigg)
\geq nj-\frac{j}{p-1}. 
$$
Since $p\neq 2$ and $n\ge 1$, thus the formal power series (\ref{0812a}) converges in $H_n$. 
This implies $H_{2n}\subset H_n^{p^n}$ and we have the assertion. 
\end{proof}

For the completion of the proof of Chapter V, Lemma 1 of \cite[p.\ 117]{L}, 
we have to check the following Lemma which is not mentioned there. 
%
%
%
%
\begin{Lem}
If ${\rm Gal}(K_1/\Q)\simeq {\rm GL}_2(\Z/p\Z)$, 
the equation $L_1\cap K_\infty=K_1$ holds. 
\label{0812c}
\end{Lem}
\begin{proof}
We put $N:=L_1\cap K_\infty$. 
Then $N$ is a Galois extension of $\Q$ and $\mbox{Gal}(K_1/\Q)$ acts on $\mbox{Gal}(N/K_1)$. 

It follows from $\mbox{Gal}(K_1/\Q)\simeq \mbox{GL}_2(\Z/p\Z)$ that 
$E[p]$ is an irreducible $\mbox{Gal}(K_1/\Q)$-module. 
Since 
$$
\Phi_1:~\mbox{Gal}(L_1/K_1)\rightarrow E[p]^r
$$
is an injective $\mbox{Gal}(K_1/\Q)$-homomorphism, 
$\mbox{Gal}(L_1/K_1)$ is a direct sum of copies of $E[p]$. 
Thus $\mbox{Gal}(N/K_1)$ is isomorphic to a direct sum of copies of $E[p]$ as a  $\mbox{Gal}(K_1/\Q)$-module.

Since $\mbox{Gal}(N/K_1)\simeq H_1/\mbox{Gal}(K_\infty/N)$ is of exponent $p$ and thus we have 
$$
H_1\supset \mbox{Gal}(K_\infty/N) \supset H_1^p.
$$
By Lemma \ref{0812b}, we have $H_1^p=H_2$, 
hence $K_1 \subset N \subset K_2$. 
Since the kernel of the surjective homomorphism 
$$
M_2(\Z_p) \rightarrow H_1/H_2:~M \mapsto 1+pM
$$
is $pM_2(\Z_p)$, we have 
$$
\mbox{Gal}(K_2/K_1)\simeq M_2(\Z/p\Z).  
$$
Via this isomorphism, 
the action of $\mbox{Gal}(K_1/\Q)\simeq\mbox{GL}_2(\Z/p\Z)$ on $M_2(\Z/p\Z)$ is

given by the conjugation of matrices. 
The $\Z/p\Z[\mbox{Gal}(K_1/\Q)]$-module $M_2(\Z/p\Z)$ has the irreducible decomposition 
$$
M_2(\Z/p\Z)^{{\rm diag}}\oplus M_2(\Z/p\Z)^{{\rm tr}=0}, 
$$
where $M_2(\Z/p\Z)^{{\rm diag}}$ means the submodule of diagonal matrices 
and $M_2(\Z/p\Z)^{{\rm tr}=0}$ means the submodule of matrices with trace zero.

$\mbox{Gal}(N/K_1)$ is the quotient of $\mbox{Gal}(K_2/K_1)$ by $\mbox{Gal}(K_2/N)$ as a $\mbox{Gal}(K_1/\Q)$-module. 
On the one hand, 
since $\mbox{Gal}(K_2/K_1) \simeq M_2(\Z/p\Z)$ is semisimple as a $\mbox{Gal}(K_1/\Q)$-module, 
the dimension over $\Z/p\Z$ of each irreducible quotient of $\mbox{Gal}(K_2/K_1)$ is 0, 1 or 3.  
On the other hand, 
since $\mbox{Gal}(N/K_1)$ is $\mbox{Gal}(K_1/\Q)$-isomorphic to a direct sum of  copies of $E[p]$, 
the dimension of its irreducible component is 0 or 2. 
Thus we have $\mbox{Gal}(N/K_1)=1$ giving $N=K_1$. 
This completes the proof.
\end{proof}
%
%
%
%
\begin{Lem}[\cite{L},~Chap.\ V, ~Lem.\ 1]
If the homomorphism 
$$\Phi_1:~{\rm Gal}(L_1/K_1) \rightarrow E[p]^r$$
is an isomorphism, then the homomorphism 
$$\Phi_\infty:~{\rm Gal}(L_\infty/K_\infty) \rightarrow T_p(E)^r$$
is an isomorphism, and for each $n$, 
$$\Phi_n:~{\rm Gal}(L_n/K_n) \rightarrow E[p^n]^r$$
is a homomorphism. 
\label{0812d}
\end{Lem}
\begin{proof}
We have a commutative diagram 
$$
\begin{array}{ccc}
\mbox{Gal}(L_\infty/K_\infty) &\rightarrow &T_p(E)^r\\
\downarrow &&\downarrow\\
\mbox{Gal}(L_1/L_1 \cap K_\infty) &\rightarrow &E[p]^r
\end{array}
$$
The vertical arrows are projections and they are surjective. 
Since $\mbox{Gal}(L_1/L_1 \cap K_\infty)=\mbox{Gal}(L_1/K_1)$ by Lemma \ref{0812c}, 
we may apply the assumption on the bottom arrow. 
Since the composite homomorphism from $\mbox{Gal}(L_\infty/K_\infty)$ to $E[p]^r$ is surjective, 
we see that the top arrow $\Phi_\infty$ is an isomorphism 
(For more detail the reader refers Chap.\ V, Lemma 2 of \cite{L}). 

For a given $n$, we have a commutative diagram 
$$
\begin{array}{ccc}
\mbox{Gal}(L_\infty/K_\infty) &\rightarrow &T_p(E)^r\\
\downarrow &&\downarrow\\
\mbox{Gal}(L_n/L_n \cap K_\infty) &\rightarrow &E[p^n]^r
\end{array}
$$
Since the vertical projections and the top arrow $\Phi_\infty$ are surjective, 
the bottom arrow is surjective. 
Since the bottom arrow is the restriction of $\Phi_n$ to $\mbox{Gal}(L_n/L_n \cap K_\infty)$, it is injective and thus an isomorphism. 

We have
$$
\sharp E[p^n]^r = \sharp \mbox{Gal}(L_n/L_n \cap K_\infty) \le 
\sharp \mbox{Gal}(L_n/K_n)\le \sharp E[p^n]^r. 
$$
Hence we have $L_n \cap K_\infty=K_n$ and $\Phi_n$ is an isomorphism. 
\end{proof}

By using Lemma \ref{0812d} and 
\begin{equation}
H^1(\mbox{Gal}(K_1/\Q),E[p])=0
\label{120812f}
\end{equation}
(cf.\ \cite[Chap.\ V,\,Thm.\ 5.1]{L}), 
the result of Bashmakov \cite{bas} is given as below. 
%
%
%
%
\begin{Thm}[\cite{L},Chap.\ V, Thm.\ 5.2, Step 1 and 2]
\label{120812g}
Assume that \\${\rm Gal}(K_1/\Q)\simeq {\rm GL}_2(\Z/p\Z)$, 
and that $A+E(\Q)_{{\rm tors}}=E(\Q)$. Then 
$\Phi_1$ is an isomorphism. Thus, 
$\Phi_n$ is an isomorphism for all $n$. 
\end{Thm}
\begin{proof}
For $P$ in $A$, 
we define the mapping 
$$
\delta_P:~\mbox{Gal}(L_1/K_1) \rightarrow E[p]:~\sigma 
\mapsto {}^\sigma T-T,
$$
where $T$ is a point in $E(L_1)$ such that $[p]_E(T)=P$. 

Clearly if $P$ is in $[p]_E(A)$, then $\delta_P$ is trivial. 
Conversely, 
let us show that $\ker\{ P \mapsto \delta_P\}=[p]_E(A)$. 
Suppose $\delta_P$ is trivial. 
Then $T$ is in $E(K_1)$ and 
$$
\mbox{Gal}(K_1/\Q) \rightarrow E[p]:~\sigma \mapsto {}^\sigma T-T
$$
is in $H^1(\mbox{Gal}(K_1/\Q),E[p])$. 
It follows from (\ref{120812f}) 
that there exists $T^\prime$ in $E[p]$ such that 
$$
{}^\sigma T-T={}^\sigma T^\prime-T^\prime
$$
for each $\sigma$ in $\mbox{Gal}(K_1/\Q)$. 
Thus we have $T-T^\prime$ is in $E(\Q)$. 
But $P=[p]_E(T-T^\prime)$, 
hence $P$ is $p$-th multiple in $E(\Q)$. 

Since $\mbox{Gal}(K_1/\Q)\simeq\mbox{GL}_2(\Z/p\Z)$, 
there are no non-trivial $p$-torsion points in $E(\Q)$. 
Since $A$ is a torsion-free module such that $E(\Q)=A+E(\Q)_{{\rm tors}}$, 
any point that is in $[p]_E(A)$ and also in $E(\Q)$ also belongs to $A$ to start with. 
Thus $T-T^\prime$ is in $A$ and $P$ is in $[p]_E(A)$. 

Hence the mapping 
$$A/[p]_E(A)\rightarrow \mbox{Hom}(\mbox{Gal}(L_1/K_1),E[p]):~P\mapsto \delta_P$$
is injective. 
On the one hand, 
consider the fundamental short exact sequence 
$E(\Q)/[p]_E E(\Q)\stackrel{\delta'}{\hookrightarrow} H^1(\mbox{Gal}(L_1/\Q),E[p])$ (cf.\ p.\ 331 of \cite{sil}). 
By (\ref{120812f}) we see that 
$H^1(\mbox{Gal}(L_1/\Q),E[p])\simeq \mbox{Hom}_{{\small \mbox{Gal}(K_1/\Q)}}(\mbox{Gal}(L_1/K_1),E[p])$. 
By the assumption the natural map $A/[p]_E(A)\stackrel{\iota}{\lra} E(\Q)/[p]_E E(\Q)$ is injective. 
Then the composite of $\iota$ and $\delta'$ is nothing but $\delta$. 
Therefore the image of $\delta$ lies in 
$\mbox{Hom}_{{\small \mbox{Gal}(K_1/\Q)}}(\mbox{Gal}(L_1/K_1),E[p])$. 

We note 
\begin{equation}
\sharp \mbox{Hom}_{{\small \mbox{Gal}(K_1/\Q)}}(\mbox{Gal}(L_1/K_1),E[p])\geq p
^r.
\label{120812h}
\end{equation}
because of rank$_{\Z/p\Z}A/[p]_E(A)=r$. 
Now the image of $\Phi_1$ is a $\mbox{Gal}(K_1/\Q)$-submodule of $E[p]^r$ 
and it is $\mbox{Gal}(K_1/\Q)$-isomorphic to a direct sum of copies of $E[p]$. 
By (\ref{120812h}) the number of elements in this sum must be $r$. 
Thus $\Phi_1$ is surjective. 
By Lemma \ref{0812d} we have the assertion. 
\end{proof}

We recall that we take $A$ such that $A+E(\Q)_{{\rm tors}}=E(\Q)$ in \S.\ 1. 
We have the following corollary from Theorem \ref{120812g}. 

\begin{Cor}
The equation $[L_n:K_n]=p^{2nr}$ holds for all $n \ge 1$. 
\label{0813a}
\end{Cor}
\section{The localization of the extension $L_n$ over $K_n$}

Fix a natural number $n$. 
Put ${\cal K}_n:=\Q_p(E[p^n])$ 
and let $\p$ be the prime ideal of ${\cal K}_n$. 
Let $A$ be a free abelian subgroup of $E({\cal K}_n)$ of rank $r\ge 1$. 
Put ${\cal L}_n:={\cal K}_n([p^n]_E^{-1}A)$ and let $\P$ be the prime ideal of ${\cal L}_n$. 

In this section we investigate the degree $[{\cal L}_n:{\cal K}_n]$. 

We denote the generators of $A$ by $P_1,\ldots,P_r$. 
For each $j\in \{1,\ldots,r\}$ we take $T_j$ in $E({\cal L}_n)$ such that $[p^n]_E(T_j)=P_j$. 
The injectivity of the mapping
$$
\Phi:~\mbox{Gal}({\cal L}_n/{\cal K}_n) \rightarrow E[p^n]^{r}~:~
\sigma \mapsto ({}^\sigma T_j- T_j)_j
$$
shows that $[{\cal L}_n:{\cal K}_n]$ divides $p^{2nr}$. 
\subsection{$P_j$, $T_j$ via Tate curve}

In this subsection we refer the reader to \cite[Ch.\ V, \S.\ 3]{Sil2} for properties on Tate curves. 

Since $E$ over $\Q_p$ has multiplicative reduction at $p$, 
there exists $q$ in $p\Z_p$ such that $E$ is isomorphic over ${\cal M}$ to the Tate curve $E_q$ for some extension ${\cal M}$ over $\Q_p$ of degree at most two. 
We see that $[{\cal M}:\Q_p]=1$ or $2$ 
according as the type of the multiplicative reduction is split or non-split. 
We denote the isomorphism from $E$ to $E_q$ by $\varphi$. 
We note that there exists a surjective $\mbox{Gal}(\bQ_p/\Q_p)$-homomorphism $\phi$ from $\bQ_p^\ast$ to $E_q(\bQ_p)$ with the kernel $q^{\Z}$. 

Since $[{\cal L}_n:{\cal K}_n]$ is a power of an odd prime $p$ 
and $[{\cal L}_n{\cal M}:{\cal K}_n]$ is at most $2$, 
we have ${\cal L}_n\cap {\cal K}_n{\cal M}={\cal K}_n$. 
Thus we have $[{\cal L}_n:{\cal K}_n]=[{\cal L}_n{\cal M}:{\cal K}_n{\cal M}]$. 
For the aim of this section it is enough to investigate the degree $[{\cal L}_n{\cal M}:{\cal K}_n{\cal M}]$. 

For each $j$ there exists an element $t_j$ in $({\cal L}_n{\cal M})^\ast$ (resp.\ $p_j$ in $({\cal K}_n{\cal M})^\ast$) such that 
$$
\varphi(T_j)=\phi(t_j)\quad (\mbox{resp.}~\varphi(P_j)=\phi(p_j)).
$$
Since $\varphi$ is an isomorphism over ${\cal M}$, 
we have
$$
{\cal L}_n{\cal M}={\cal K}_n{\cal M}(T_1,\ldots,T_{r})={\cal K}_n{\cal M}(\varphi(T_1),\ldots,\varphi(T_{r})).
$$
Since $\phi$ is a $\mbox{Gal}(\bQ_p/\Q_p)$-homomorphism, 
we have 
$$
{\cal K}_n{\cal M}(\varphi(T_1),\ldots,\varphi(T_{r}))=
{\cal K}_n{\cal M}(\phi(t_1),\ldots,\phi(t_{r}))\subset 
{\cal K}_n{\cal M}(t_1,\ldots,t_{r}) \subset 
{\cal L}_n{\cal M}.
$$
Thus we have 
$$
{\cal L}_n{\cal M}={\cal K}_n{\cal M}(t_1,\ldots,t_{r}). 
$$

Furthermore we have 
$$
\varphi(P_j)=\varphi([p^n]_E(T_j))=[p^n]_{E_q}(\varphi(T_j))=
[p^n]_{E_q}(\phi(t_j))=\phi(t_j^{p^n}).
$$
Thus we have
$$
t_j^{p^n}=p_j q^{m}
$$
for some $m$ in $\Z$. Since $E_q[p^n]=\langle \zeta_{p^n},q^{\frac{1}{p^n}}\rangle\subset E_q({\cal K}_n{\cal M})$, we have 
$$
{\cal K}_n{\cal M}(t_1,\ldots,t_{r})=
{\cal K}_n{\cal M}(p_1^{\frac{1}{p^n}},\ldots,
p_{r}^{\frac{1}{p^n}}). 
$$
Thus we have $[{\cal L}_n{\cal M}:{\cal K}_n{\cal M}]$ divides $p^{nr}$ and so does $[{\cal L}_n:{\cal K}_n]$.  
%
%
%
%
\begin{Lem}
Assume that $E$ has multiplicative reduction at $p$. 
Then, 
the equation ${\cal L}_n{\cal M}={\cal K}_n{\cal M}(p_1^{\frac{1}{p^n}},\ldots,p_{r}^{\frac{1}{p^n}})$ holds. 
Specially, 
$[{\cal L}_n:{\cal K}_n]$ divides $p^{nr}$
\label{1228c}
\end{Lem}
\subsection{The case of $A\subset E(\Q_p)$}

In this subsection we assume that $A \subset E(\Q_p)$ and 
we modify Lemma \ref{1228c}. 

At first we consider the case of split multiplicative reduction. 
Then we have $p_j \in \Q_p^\ast$ and ${\cal M}=\Q_p$. 
%
%
%
%
\begin{Lem}
Assume that $E$ has split multiplicative reduction at $p$ 
and $A \subset E(\Q_p)$. 
If $p \nmid \mbox{ord}_p(q)$, 
we have ${\cal L}_n \subset {\cal K}_n(\sqrt[p^n]{1+p})$ holds. 
Namely ${\cal L}_n/{\cal K}_n$ is a cyclic extension 
and $[{\cal L}_n:{\cal K}_n]$ divides $p^n$. 
\label{0110a}
\end{Lem}
\begin{proof}
By  Lemma \ref{1228c}, 
we have
$$
{\cal K}_n(T_1,\ldots,T_r)=
{\cal K}_n(p_1^{\frac{1}{p^n}},\ldots,p_r^{\frac{1}{p^n}}) 
$$ 
for some $p_j$ in $\Q_p^\ast$. 
It follows from 
$$
\Q_p^\ast \simeq \langle p \rangle \times \mu_{p-1}\times (1+p\Z_p) 
\simeq \Z \times \Z/(p-1)\Z \times \Z_p
$$
that
$$
\Q_p^\ast/(\Q_p^\ast)^{p^n}=\langle p\rangle\times 
\langle1+p\rangle\simeq (\Z/p^n\Z)^2. 
$$
It follows from $p \nmid \mbox{ord}_p(q)$ 
that we can replace the generator $p$ by $q$, namely, 
$$
\langle p\rangle\times 
\langle1+p\rangle \simeq \langle q\rangle\times 
\langle1+p\rangle.
$$
Since $q^{\frac{1}{p^n}}$ is in ${\cal K}_n$, 
we have 
$$
{\cal K}_n(T_1,\ldots,T_r)\subset 
{\cal K}_n(\sqrt[p^n]{1+p}). 
$$
This completes the proof. 
\end{proof}

Secondly we consider the case of non-split multiplicative reduction. 
Then we have $p_j \in \{ x \in {\cal M}^\ast~|~{\rm N}_{{\cal M}/\Q_p}(x)\in q^{\Z}\}$ 
and ${\cal M}$ is an unramified quadratic extension over $\Q_p$ 
(cf.\ \cite{sil}, p.\ 357, Thm.\ 14.1). 
We put 
$$
H:=\{x \in {\cal M}^\ast~|~{\rm N}_{{\cal M}/\Q_p}(x) \in q^{\Z}\}.
$$
We discuss about the group structure of $H/H^{p^n}$. 

Since ${\cal M}$ is an unramified extension of $\Q_p$, 
there exists no non-trivial $p$-th power root of unity in ${\cal M}$. 
We denote the unit group of ${\cal M}$ (resp.\ $\Q_p$) by $U_{{\cal M}}$ (resp.\ $U_{\Q_p}$). 
Since ${\cal M}$ is unramified over $\Q_p$, 
we have the exact sequence 
\begin{equation}
1 \rightarrow U_{{\cal M},1} \rightarrow U_{{\cal M}} \mapright{{\rm N}_{{\cal M}/\Q_p}} U_{\Q_p} 
\rightarrow 1
\label{120318b}
\end{equation}
for a subgroup $U_{{\cal M},1}$ of $U_{{\cal M}}$. 
The group $U_{{\cal M},1}$ is a subgroup of $H$. 

Since ${\cal M}$ is unramified over $\Q_p$, 
the order $\mbox{ord}_p({\rm N}_{{\cal M}/\Q_p}(x))$ is even 
for any $x$ in ${\cal M}^\ast$. 
Thus, 
if $\mbox{ord}_p(q)$ is odd, 
there is no element $x$ in ${\cal M}$ such that ${\rm N}_{{\cal M}/\Q_p}(x)=q$. Since ${\rm N}_{{\cal M}/\Q_p}(q)=q^2$, 
we have $H= q^{\Z} \times U_{{\cal M},1}$. 

If $\mbox{ord}_p(q)=\nu$ is even, 
then $qp^{-\nu}$ is in $U_{\Q_p}$ 
and there exists $u$ in $U_{{\cal M}}$ such that 
${\rm N}_{{\cal M}/\Q_p}(u)=qp^{-\nu}$. 
We have ${\rm N}_{{\cal M}/\Q_p}(up^{\frac{\nu}{2}})=q$ 
and thus $H=\langle up^{\frac{\nu}{2}} \rangle \times U_{{\cal M},1}$. 

We put 
\begin{equation}
H_0:=q^{\Z} \times U_{{\cal M},1}.
\label{120318a}
\end{equation}
Then the index $[H:H_0]=1, 2$ 
according as the parity of $\mbox{ord}_p(q)$. 
If $[H:H_0]=2$, 
there exists an element $t$ in $H$ such that ${\rm N}_{{\cal M}/\Q_p}(t)=q$. 
Either $x$ or $xt^{p^h}$ is in $H_0$. 
Thus we have 
$$
H/H^{p^n}\simeq H_0/H_0^{p^n}. 
$$
Therefore, we have 
$$
H/H^{p^n}\simeq q^{\Z}/q^{p^n\Z} \times U_{{\cal M},1}/U_{{\cal M},1}^{p^n}
$$
for any parity of $\mbox{ord}_p(q)$. 

We see that 
\begin{eqnarray*}
U_{{\cal M}}&\simeq& \mu_{p^2-1} \times U_{{\cal M}}^{(1)} \simeq \Z/(p^2-1)\Z\times 
\Z_p^{\oplus 2},\\
U_{\Q_p}&\simeq& \mu_{p-1} \times U_{\Q_p}^{(1)} \simeq \Z/(p-1)\Z\times 
\Z_p,
\end{eqnarray*}
where $U_{{\cal M}}^{(1)}$ (resp.\ $U_{\Q_p}^{(1)}$) is the principal congruence subgroup modulo $p$. 
It follows from (\ref{120318b}) that 
$$
U_{{\cal M},1}\simeq \Z/(p+1)\Z \times \Z_p
$$
and 
$$
U_{{\cal M},1}/U_{{\cal M},1}^{p^n}\simeq \Z/p^n\Z.
$$
Therefore we have
$$
H/H^{p^n}\simeq \langle q \rangle \times \langle u \rangle
\simeq \Z/p^n\Z \times \Z/p^n\Z
$$
for some $u$ in $U_{{\cal M},1}$. 

It follows from 
$E[p^n]=\langle\zeta_{p^n},q^{\frac{1}{p^n}}\rangle \subset E({\cal K}_n{\cal M})$ that 
$$
{\cal L}_n{\cal M} \subset {\cal K}_n{\cal M}(u^{\frac{1}{p^n}}).
$$
Thus ${\cal L}_n{\cal M}/{\cal K}_n{\cal M}$ is 
a cyclic extension 
and the degree $[{\cal L}_n{\cal M}:{\cal K}_n{\cal M}]$ divides $p^n$. 
Since [${\cal K}_n{\cal M}:{\cal K}_n$] is coprime to $[{\cal L}_n:{\cal K}_n]$, we have the following lemma. 
%
%
%
%
\begin{Lem}
Assume that $E$ has non-split multiplicative reduction at $p$,  and 
$A \subset E(\Q_p)$. 
Then ${\cal L}_n/{\cal K}_n$ is a cyclic extension and 
$[{\cal L}_n:{\cal K}_n]$ divides $p^n$. 
\label{120317c}
\end{Lem}

Since $\mbox{ord}_p(q)=\mbox{ord}_p(\Delta)$, 
the assumption $p \nmid \mbox{ord}_p(q)$ is equivalent to 
$p \nmid \mbox{ord}_p(\Delta)$. 
If the conductor of $E$ is a prime $p$, 
then $p \ge 11$ and $\Delta$ divides $p^5$. 
Thus the assumption $p\nmid \mbox{ord}_p(q)$ is satisfied 
for elliptic curves with prime conductor $p$. 
\section{Proof of Theorem \ref{1002a}}
Let notations and assumptions be the same as in Theorem \ref{1002a}. 
As we see in the end of the previous section, 
we have $p \nmid \mbox{ord}_p(q)$. 
By using Lemmas \ref{0110a} and \ref{120317c}, Corollary \ref{0813a}, 
we estimate the degree $[L_n \cap K^{{\rm ur}}_n:K_n]$.

Let $\p$ be a prime ideal of $K_n$ lying above $p$ 
and let $\P$ be a prime ideal of $L_n$ lying above $\p$. 
Let $I_\P$ be the inertia group of $\P$ in $\mbox{Gal}(L_n/K_n)$. 
Since $L_n/K_n$ is abelian, $I_\P$ is independent of a choice of $\P$ lying above $\p$. 
We put 
$$
I:=\langle \sigma I_\P\sigma^{-1}~|~\sigma \in \mbox{Gal}(L_n/\Q) \rangle.
$$
The group $\sigma I_\P\sigma^{-1}$ is the inertia group of ${}^\sigma \P$ lying above ${}^\sigma\p$. 
Since $L_n/K_n$ is abelian, 
$$
I=\langle \sigma I_\P\sigma^{-1}~|~\sigma \in \mbox{Gal}(K_n/\Q) \rangle.
$$
Since $\mbox{Gal}(L_n/K_n)$ is a normal subgroup of $\mbox{Gal}(L_n/\Q)$, 
$I$ is contained in $\mbox{Gal}(L_n/K_n)$ 
and its fixed field $L_n^I$ is unramified over $K_n$. 

It follows from $A\subset E(\Q)$ 
that $\Phi_n$ is a $\mbox{Gal}(K_n/\Q)$-isomorphism. 
$$
\Phi_n(I)=\langle {}^\sigma \Phi_n(I_\P)~|~\sigma \in \mbox{Gal}(K_n/\Q) \rangle\subset E[p^n]^r.
$$
By using Lemmas \ref{0110a} and \ref{120317c}, 
$I_\P$ is cyclic. 
We denote its generator by $x$. 
Then we have 
$$
\Phi_n(I)=\langle {}^\sigma \Phi_n(x)~|~\sigma \in \mbox{Gal}(K_n/\Q) \rangle
=(\Z/p^n\Z)[\mbox{Gal}(K_n/\Q)]\Phi_n(x).
$$
Since $(\Z/p^n\Z)[\mbox{Gal}(K_n/\Q)]\simeq M_2(\Z/p^n\Z)$, 
we have 
$$
\sharp I =\sharp \Phi_n(I) \leq p^{4n}. 
$$
By using Corollary \ref{0813a}, 
we have 
\begin{equation}
[L \cap K^{{\rm ur}}_n:K_n]=\frac{[L_n:K_n]}{[L_n:L_n^I]}
\geq \frac{p^{2nr}}{p^{4n}}=p^{(2r-4)n}
\label{120307e}
\end{equation}
for each $n\geq 1$. 

This completes Theorem \ref{1002a}.
%
%
%
%
\begin{Exa}
Let $E$ be the elliptic curve defined by $y^2+y=x^3-7x+6$. 
This elliptic curve has prime conductor $p=5077$ with $j(E)=-\frac{2^{12}\cdot 3^3\cdot 7^3}{p}$ 
and it has non-split multiplicative reduction at $p$. 
By {\rm \cite{bgz}}, 
we see that $E(\Q)$ is a free $\Z$-module of rank $3$ and it is generated by 
$P_1=(0,2),\ P_2=(1,0),\ P_3=(2,0)$. 
Then by Theorem {\rm \ref{1002a}}, we have $\kappa_n\ge n(2\cdot 3-4)=2n$. 
Hence the class number $h_{\Q(E[p^n])}$ satisfies 
$$5077^{2n}~|~h_{\Q(E[p^n])}$$
for each $n\ge 1$. 
\end{Exa}
\section{A note on class number formula obtained from cyclotomic 
$\Z_p$-extension}

Let the notation be the same as in \S 1. 
As we see in the beginning of \S.\ 2, 
we have $\mbox{Gal}(K_1/\Q)\simeq \mbox{GL}_2(\Z/p\Z)$. 
We denote the maximal unramified abelian extension of $\Q(\zeta_{p^n})$ by $\Q(\zeta_{p^n})^{{\rm ur}}$.
%
%
%
%
\begin{Lem}
$\Q(\zeta_{p^n})^{{\rm ur}} \cap K_n=\Q(\zeta_{p^n})$ 
holds for each positive integer $n$. 
\label{0324a}
\end{Lem}
\begin{proof}
The image of the surjective homomorphism 
\begin{equation}
G:={\rm Gal}(K_n/\Q(\zeta_{p^n}))\lra {\rm Gal}(\Q(\zeta_{p^n})^{{\rm ur}}\cap K_n/\Q(\zeta_{p^n}))
\label{120327a}
\end{equation}
is commutative. 
Thus the homomorphism (\ref{120327a}) factors the abelian quotient $G^{{\rm ab}}=G/[G,G]$. 
However, 
we have $[G,G]=G$ by $G\simeq {\rm SL}_2(\Z/p^n\Z)$ (cf.\ p.\ 1559, Lemma 19 of \cite{J}). 
Thus the image of the homomorphism (\ref{120327a}) is trivial. 

This completes the proof. 
\end{proof}
%
%
%
%
\begin{Lem}
$L_n \cap \Q(\zeta_{p^n})^{{\rm ur}}K_n=K_n$ 
holds for each positive integer $n$. 
\label{0324b}
\end{Lem}
\begin{proof}
By Lemma \ref{0324a}, 
we have
$$\mbox{Gal}(\Q(\zeta_{p^n})^{{\rm ur}}K_n/\Q(\zeta_{p^n}))
\simeq \mbox{Gal}(K_n/\Q(\zeta_{p^n}))\times \mbox{Gal}(\Q(\zeta_{p^n})^{{\rm ur}}/\Q(\zeta_{p^n})).$$ 

Recall that $\mbox{Gal}(K_n/\Q)$ acts on $\mbox{Gal}(L_n/K_n)$ by the usual manner, since $\mbox{Gal}(L_n/K_n)$ is abelian. 
Thus the action of $\mbox{Gal}(K_n/\Q(\zeta_{p^n}))$ 
on $\mbox{Gal}(L_n \cap \Q(\zeta_{p^n})^{{\rm ur}}/K_n)$ 
induced by (\ref{120327a}) is trivial. 
Then for each $\sigma$ in $\mbox{Gal}(L_n \cap \Q(\zeta_{p^n})^{{\rm ur}}K_n/K_n)$, 
the point ${}^\sigma T_j-T_j\in E[p^n]$ is a $\mbox{Gal}(K_n/\Q(\zeta_{p^n}))$-fixed point. 
Since $\mbox{Gal}(K_n/\Q(\zeta_{p^n}))\simeq \mbox{SL}_2(\Z/p^n\Z)$ contains $-1$ 
and $p$ is an odd prime number, 
the point ${}^\sigma T_j-T_j$ is never fixed by  $\mbox{Gal}(K_n/\Q(\zeta_{p^n}))$ if $\sigma\neq 1$, 
hence we have $\mbox{Gal}(L_n \cap \Q(\zeta_{p^n})^{{\rm ur}}K_n/K_n)=1$. 
\end{proof}

By Lemma \ref{0324b}, the contribution of the class number that is shown in Theorem \ref{1002a} does not come from the class number of $\Q(\zeta_{p^n})$.

\end{document}